\documentclass[11pt]{article}
\usepackage[doc]{optional}
\usepackage{enumitem}
\usepackage{pstricks}
\usepackage{mathtools,esvect}
\usepackage{cite}
\usepackage{color}
\usepackage{float}
\usepackage{subcaption}
\usepackage{soul}
\definecolor{labelkey}{rgb}{0,0.08,0.45}
\definecolor{refkey}{rgb}{0,0.6,0.0}

\usepackage[left]{lineno}
\usepackage{blindtext}

\usepackage{mathpazo}
\usepackage{amsmath}
\usepackage{amssymb}
\usepackage{theorem}
\usepackage{fancyhdr}

\usepackage[margin=1.0in]{geometry}

\newcommand{\fenv}[1]%
{\ensuremath{\,\overrightarrow{\operatorname{env}}_{#1}}}
\newcommand{\benv}[1]%
{\ensuremath{\,\overleftarrow{\operatorname{env}}_{#1}}}

\newcommand{\RR}{\ensuremath{\mathbb R}}

\newcommand{\NN}{\ensuremath{\mathbb N}}

{\begin{list}{}{%
\settowidth{\labelwidth}{\textrm{#1~}}%
\setlength{\leftmargin}{\labelwidth+\labelsep}}}
{\end{list}}
\newtheorem{theorem}{Theorem}[section]
\newtheorem{lemma}[theorem]{Lemma}
\newtheorem{corollary}[theorem]{Corollary}
\newtheorem{proposition}[theorem]{Proposition}
\newtheorem{definition}[theorem]{Definition}
\theoremstyle{plain}{\theorembodyfont{\rmfamily}
}
\theoremstyle{plain}{\theorembodyfont{\rmfamily}
}
\theoremstyle{plain}{\theorembodyfont{\rmfamily}
}
\theoremstyle{plain}{\theorembodyfont{\rmfamily}
\newtheorem{example}[theorem]{Example}}
\newtheorem{fact}[theorem]{Fact}
\theoremstyle{plain}{\theorembodyfont{\rmfamily}
\newtheorem{remark}[theorem]{Remark}}

\def\doi{DOI}

\newcounter{count}

\allowdisplaybreaks

\begin{document}

\title{\textrm{An inexact strategy for the projected gradient algorithm in vector optimization problems on variable ordered spaces}}
\author{
J.Y.\ Bello-Cruz\thanks{Corresponding author. Department of Mathematical Sciences, Northern Illinois University. Watson Hall 366, DeKalb, IL, USA - 60115. E-mail:
\texttt{yunierbello@niu.edu.}}~~~~~~~~G. Bouza Allende\thanks{Facultad de Matem\'atica y Computaci\'on, Universidad de La Habana, La Habana,   Cuba - 10400.  E-mail:
\texttt{gema@matcom.uh.cu} }}

\date{August, 2019}

\maketitle \thispagestyle{fancy}

\begin{abstract} \noindent 
Variable order structures model situations in which the comparison between two points depends on a point-to-cone map. In this paper, an inexact projected gradient method for solving smooth constrained vector optimization problems on variable ordered spaces is presented. It is shown that every accumulation point of the generated sequence satisfies the first order necessary optimality condition. The convergence of all accumulation points to a weakly efficient point is established under suitable convexity assumptions for the objective function. The convergence results are also derived in the particular case in which the problem is unconstrained and if exact directions are taken as descent directions. Furthermore, we investigate the application of the proposed method to optimization models where the domain of the variable order map and the objective function are the same. In this case, similar concepts and convergence results are presented. Finally, some computational experiments designed to illustrate the behavior of the proposed inexact methods versus the exact ones (in terms of CPU time) are performed. \end{abstract}



\section{Introduction}
Variable order structures are a natural extension of the well-known fixed (partial) order given by a closed, pointed and convex cone; see  \cite{gabrielle-book}. These kind of orderings model situations in which the comparison between two points depends on a set-valued map. 
These problems have recently received much attention from the optimization community due to their broad application to several different areas. 
{Variable order structures (VOS), given by a point-to-cone valued map, were well studied in \cite{gabrielle-book, gabi2, ap13}, motivated by important applications.  VOS appear in medical diagnosis \cite{Eichfelder-med},
portfolio optimization \cite{ap31}, capability theory of well-being \cite{bao-mord-soub2}, psychological modeling \cite{bao-mord-soub}, consumer preferences \cite{John1, John2} and location theory, etc; see, for instance, \cite{ap2, ap13}.} The main goal of these models is to find an element of a certain set such that the evaluation of the objective function cannot be improved by the image of any other feasible point with respect to the variable order. So, their mathematical description corresponds to the so-called Optimization Problem(s) on Variable Ordered Spaces (OPVOS(s)).   {For the reasons mentioned above, although the variable order setting is a relatively new avenue of research, several papers and even books have been published with many real-life problems modeled via this approach. For example, an interesting application of vector optimization with a variable structure is given in the theory of consumer demand in economics by John \cite{John1,John2}. These papers present a local and a global theory in order to explain consumer behaviors. In the local approach it is assumed that the consumer faces a nonempty set of feasible alternatives, $A \subset \mathbb{R}^n$. By contrast with the global approach, a local preference only requires that the consumer is able to rank alternatives in a small neighborhood of a given commodity bundle relative to that bundle. This idea can be represented by an economical comparative function $g \colon \mathbb{R}^n\to \mathbb{R}^n$ such that 
$\overline{y}$ in a neighborhood of $y$ is interpreted to be better than $y$ if and only if $g(\overline{y})^T(y-\overline{y}) < 0$. The choice set assigned to $A$ in the local theory is then given by 
$$C(A) := \left\{\overline{y} \in A : \forall y \in A,\; g(\overline{y})^T(y -\overline{y}) \geq 0\right\}.$$
This leads to a set-valued map $K \colon \mathbb{R}^n \rightrightarrows  \mathbb{R}^n$  defined by 
$$K(\overline{y}) := \left\{d \in \mathbb{R}^n : g(\overline{y})^Td \geq 0\right\}.$$
If the consumer is interested in an alternative $\overline{y}\in A$ such that
$\forall y \in A$, $g(\overline{y})^T(y-\overline{y}) \geq 0$,
then $y -\overline{y}\in K(\overline{y})$ for all $y \in A$, i.e., $ A\subseteq \{\overline{y}\} + K(\overline{y}).$
Furthermore, if the consumer is looking for alternatives $\bar y\in A$ such that
$$\forall y \in A \setminus \{\overline{y}\},\; g(\overline{y})^T(y -\overline{y}) < 0,$$
then, for all $y \in A \setminus \{\overline{y}\}$,
$y -\overline{y} \notin K(\overline{y}).$
This means that the consumer is looking for alternatives $\bar y \in A$ such that
$A \setminus (\overline{y}-  K(\overline{y})) = \{\overline{y}\},$ 
i.e., the consumer is looking for minimal points of the  vector optimization
problem with variable domination structure $\min_K y \; s.t. \;  y\in A.$ }

OPVOSs have been studied in \cite{ap11}, in the sense of finding a minimizer of the image of a vector function, with respect to a variable ordered structure depending on points
in the image.  It is a particular case of the problem described in \cite{gabi2}, where the goal of the model is to find a minimum of a set.  
Here we will consider a partial (variable) order defined by the cone-valued map which is used to define our problem.  We want to point out that  OPVOSs  generalize the classical vector optimization problems. Indeed, they correspond to the case in which the order is defined by a constant cone valued map.  Many approaches have been proposed to solve the classical constrained vector optimization, such as projected gradient methods, proximal points iterations, weighting technique schemes, Newton-like and subgradient methods; see, for instance, \cite{yunier-2013,yunier-luis-2014,BIS,grana-maculan-svaiter,fliege-grana-svaiter,jahn0, [21], luis-jef-yun, [20], luc1,fliege-svaiter,grana-svaiter}.  It is worth noting that, as far as we know, only a few of these schemes mentioned above have been proposed and studied in the variable ordering setting; as, e.g., the steepest descent algorithm and sub-gradient-like algorithm for unconstrained problems, and  {a Newton-like method}; see, for instance,  \cite{bento-gema-2018,nous, luis-gema-yunier-2014}. The use of extensions of these iterative algorithms to the variable ordering setting is currently a promising idea. So, it is important to find efficient solution algorithms for solving these kinds of models.

In this work,
we present the projected gradient method with an inexact strategy for solving constrained variable order vector problems because of its simplicity and the adaptability to the vector structure of the  problem.
Moreover, we derive the convergence of the exact projected gradient method and the inexact projected gradient method for the unconstrained problem. Finally, analogous results are obtained if the variable order is given by a point-to-cone map whose domain coincides with the image of the objective function.

This work is organized as follows. The next section provides some notations and preliminary results that will be used in the remainder of this paper. We also recall the concept of $K$--convexity of a function on a  variable ordered space and present
some properties of this class. Section $3$ is devoted to the presentation of {the inexact projected gradient algorithm}.
The convergence of the sequence generated by the projected gradient method is shown in Section $4$. Then, under the $K$--convexity of the objective function and the convexity of the set of feasible solutions, we guarantee that the generated sequence is bounded and all its accumulation points are
solutions of the variable order problem. Section $5$ discusses the properties of this algorithm when the variable order is taken as a cone-value set from the image of the objective function. Section $6$ introduces some examples illustrating the behavior of both proposed methods. Finally, some final remarks are given.

\section{Preliminaries}
In this section {we present some preliminary results and definitions.}
First we introduce some useful notations. Throughout
this paper, $p:=q$ indicates that $p$ is defined to be
equal to $q$ and we
write $\NN$ for the nonnegative integers $\{0, 1, 2,\ldots\}$. The inner product in $\RR^n$ will be denoted by $\langle\cdot, \cdot \rangle$ and the induced norm by $\|\cdot\|$.  The closed ball centered at $x$ with radius $r>0$ is represented by
$\mathbb{B}(x,r):=\{y\in\RR^n:{\mbox{dist}(x,y):=\|y-x\|}\leq r\}$ and also the sphere by  $\mathbb{S}(x,r):=\{y\in\mathbb{B}(x,r):\mbox{dist}(x,y)= r\}$.  Given two bounded sets $A$ and $B$, we will consider
$\mbox{d}_H(A,B)$ as the {\em Hausdorff} distance, \emph{i.e.} $$\mbox{d}_H(A,B) := \max\left\{\,\sup_{a \in A} \inf_{b \in B} \mbox{dist}(a,b),\, \sup_{b \in B}
\inf_{a \in A} \mbox{dist}(a,b)\,\right\},$$ or equivalently
$
\mbox{d}_H(A,B)=\inf\{\epsilon\ge0 : A\subseteq B_\epsilon\quad  \mbox{and}\quad B\subseteq A_\epsilon \},
$ where $$D_\epsilon:=\displaystyle\cup_{ d\in D}\{x\in\RR^n : \mbox{dist}(d,x)\le \epsilon\}$$ is the $\epsilon$--enlargement of any set $D$. 
The set $D^c$ and ${\rm int}(D)$ denote the complement and the interior of of $D$, respectively. {{${\rm conv}(D)$ is used for the convex hull of $D$, i.e., the intersection of all convex sets containing $D$. If $D$ is closed and convex, we define the orthogonal projection of $x$ onto $D$, denoted by $P_D(x)$, as the unique point in $D$ such that $\|P_D(x)-y\| \le \|x-y\|$ for all $y \in D$}. Given the partial order structure induced by a cone $\mathcal{K}$,
the concept of infimum of a sequence can be defined. Indeed, for a
sequence $(x^k)_{k\in\NN}$ and a cone $\mathcal{K}$, the point
$x^*$ is $\inf_k \{x^k\}$ iff $(x^k-x^*)_{k\in\NN}\subset
\mathcal{K}$, and there is not $x$ such that {$x-x^*\in \mathcal{K}$, $x\neq x^*$}
and $(x^k- x)_{k\in\NN}\subset \mathcal{K}$. We said that $\mathcal{K}$ has the {\em Daniell} property if for all sequence $(x^k)_{k\in\NN}$ such that $(x^k-x^{k+1})_{k\in\NN}\subset \mathcal{K}$ and for some $\hat{x}$, $(x^k-\hat x)_{k\in\NN}\subset \mathcal{K}$,
then $\lim_{k\rightarrow\infty}x^k=\inf_k \{x^k\}$.
Here we assume that $K(x)$, $x\in\RR^n$, is a convex, pointed, and closed cone, which guarantees that $K(x)$ has  the Daniell property as was shown in \cite{theluc}.
For each $x\in \mathbb{R}^n$,  the dual cone of $K(x)$ is defined as
$K^*(x):=\{w\in \mathbb{R}^m: \langle w,y\rangle\geq 0, \text{ for all }y\in K(x)\}$.
As usual, the graph of a {set-valued map $K\colon \RR^n\rightrightarrows\RR^m$} is the set $Gr(K):=\{(x,y)\in \mathbb{R}^n\times \mathbb{R}^m:\; y\in K(x)\}.$ Finally, we remind that the mapping $K$ is
closed  if
$Gr(K)$ is a closed subset of $\mathbb{R}^n\times \mathbb{R}^m$.

Next,  we will define the constrained  vector optimization problem on variable ordered spaces, which finds  a $K$--minimizer of the vector function $F\colon\RR^n\to \RR^m$ in the set $C$ as
\begin{equation}\label{(P)}K-\min F(x),\quad x\in C.\end{equation}
Here $C$ is a nonempty convex and closed subset of  $\RR^n$ and $K\colon\mathbb{R}^n\rightrightarrows \mathbb{R}^m$ is a point-to-cone map, where for each $x\in \mathbb{R}^n${, 
$K(x)$} is a {pointed, convex and closed cone with nonempty interior}. We say that the  point $x^*\in C$ is a minimizer of problem \eqref{(P)} if for all $x\in C$,  $$F(x)-F(x^*)\notin {-}K(x^*)\setminus\{0\}.$$
The set of all minimizers {(or efficient solutions)} of problem \eqref{(P)} is denoted by  $S^*$.

As in the case of classical vector optimization, related solution concepts  such as weakly efficient and stationary  points can be extended to the constrained setting.
The point $x^*\in C$ is a weak solution of problem \eqref{(P)} {iff} for all $x\in C$, $F(x)-F(x^*)\notin -{\rm int}(K(x^*))$, $S^w$ is the set of all weak solution points.
We want to point out that this definition corresponds with the concept of weak minimizer given in \cite{gabi2}. On the other hand, if $F$ is a continuously  differentiable function, the point  {$x^*\in C$} is stationary, iff  for all $d\in  C-x^*:=\{d\in\mathbb{R}^n:\; d=c-x^*,\text{ for some }c\in C\}$, we have
\begin{equation}\label{stationary-inclusion} J_F(x^*)d\notin -{\rm int}(K(x^*)),\end{equation} where $J_F$ denotes the Jacobian {matrix} of $F$.
The set of all stationary points will be denoted by $S^{s}$.

Now we present a version of Proposition~2.1 of \cite{nous}{, which is an extension of Lemma $5.2$ of \cite{[20]} for constrained OPVOS.} 
\begin{proposition}\label{primal}
 Let  {$x^*\in C$} be a weak solution of problem \eqref{(P)}.  If $F$ is a continuously differentiable function, then $x^*$ is a stationary point.
\end{proposition}
\begin{proof}  Suppose that   {$x^*\in C$} is a weak solution of problem \eqref{(P)}. Fix  $d\in C-x^*$. By definition there exists $c\in C$, such that $d=c-x^*$. 
Since $C$ is a convex set,  for all $\alpha\in[0,1]$, $x^*+\alpha d\in C.$   Since  {$x^*\in C$} is a weak solution of problem \eqref{(P)},  $F(x^*+\alpha d)-F(x^*)\notin -{\rm int}(K(x^*))$. Hence,
\begin{equation}\label{**}F(x^*+\alpha d)-F(x^*)\in (-{\rm int}(K(x^*)))^c.
 \end{equation}
The Taylor expansion of $F$ at $x^*$ leads us to  $F(x^*+\alpha d)=F(x^*)+\alpha  J_F(x^*)d +o(\alpha).$  The last equation together with \eqref{**} implies
$\alpha  J_F(x^*)d +o(\alpha)\in (-{\rm int}(K(x^*)))^c.$
Using that $(-{\rm int}(K(x^*)))^c$ is a closed cone, and since $\alpha>  0$, it follows that   $$  J_F(x^*)d +\frac{o(\alpha)}{\alpha}\in (-{\rm int}(K(x^*)))^c.$$
Taking limit in the above inclusion, when $\alpha$ goes to $0$, and using the closedness of $(-{\rm int}(K(x^*)))^c$, we obtain that $  J_F(x^*)d\in (-{\rm int}(K(x^*)))^c,$
establishing that $x^*\in S^s$.\end{proof}

In classical optimization, {stationarity is also a sufficient condition for weak} minimality {under convexity assumptions}. For  vector optimization problems on variable ordered spaces, the convexity concept was introduced in Definition $3.1$ of \cite{nous} as follows:
\begin{definition}
 We say that  $F$ is a $K$--convex function {on} $C$ if for all $\lambda\in [0,1]$, $x,\bar{x}\in C$,
 $$F(\lambda x+(1-\lambda)\bar{x})\in \lambda F(x)+(1-\lambda)F(\bar{x})- K(\lambda x+(1-\lambda)\bar{x}).$$
\end{definition}
It is worth noting that in the variable order setting the convexity of {${\rm epi}{(F)}:=\{(x,y)\in\RR^n\times\RR^m \,|\, F(x)\in y - K(x)\}$} is equivalent to the $K$--convexity of $F$ iff $K(x)\equiv K$ for all $x\in \RR^n$; see Proposition $3.1$ of \cite{nous}.
As already shown in \cite{nous}, $K$--convex functions have directional derivatives under natural assumptions; see Proposition $3.5$ of \cite{nous}.  In particular, if $Gr(K)$ is closed and $F\in \mathcal{C}^1$ is $K$--convex, then we have the gradient inclusion inequality as follows: $$F(x)-F(\bar{x})\in  J_F(\bar{x})(x-\bar{x})+K(\bar{x}), \quad  x, \bar{x}\in C.$$ 
In the next proposition, we study the relation between stationarity, descent directions and weak solution concept in the constrained sense for problem \eqref{(P)} {extending to the variable order setting the results presented in Proposition $1$ of \cite{[21]} and Lemma 5.2 of \cite{[20]}.}

\begin{proposition}\label{note1}
 Let $K$ be a point-to-cone and closed mapping, and $F\in \mathcal{C}^1$ be a $K$--convex function. Then:
 \item[ {\bf(i)}] The point  {$x^*\in C$} is a weak solution of problem \eqref{(P)} iff it is a stationary point.
 \item[ {\bf(ii)}] If for all $d\in C-x^*$, $ J_F(x^*)d\notin  -K(x^*)\setminus\{0\}$, then $x^*$ is a minimizer of problem \eqref{(P)}.
 \end{proposition}
\begin{proof} {\bf (i):} Let $x^*\in S^s$, where $S^s$ is the set of the stationary points.  If  {$x^*\in C$} is not a weak minimizer then there exists $x\in C$ such that
$- k_1:=F(x)-F(x^*)\in -{\rm int}(K(x^*))$.
By the convexity  of $F$, for some $k_2\in K(x^*)$, we have 
\begin{equation*} -k_1=F(x)-F(x^*)=  J_F(x^*)(x-x^*)+k_2.\end{equation*} It follows from {the above equality} that  
\begin{equation}
\label{stationarity}
 J_F(x^*)(x-x^*)=-(k_1+k_2).\end{equation}
Moreover, since  $ K(x^*)$ is a convex cone,  $k_1\in {\rm int}(K(x^*))$ and $k_2\in K(x^*)$, it holds that $k_1+k_2\in {\rm int}(K(x^*))$. 
Thus, {the last two equalities} imply that $  J_F(x^*)(x-x^*)\in - {\rm int}(K(x^*))$, which contradicts the fact that $x^*$ is a stationary point because $x$ belongs to $C$ and hence $x-x^*\in C-x^*$.
The conversely implication was already shown in Proposition \ref{primal}.
\medskip

\noindent {\bf (ii):} By contradiction suppose that there exists $x\in C$ such that $F(x)-F(x^*)=-k_1,$ where $k_1\in K(x^*)\setminus \{0\}.$
Combining the previous condition with \eqref{stationarity}, it follows that  $$  J_F(x^*)(x-x^*)=-(k_1+k_2)\in -K(x^*).$$
Using that $ J_F(x^*)(x-x^*)\notin -K(x^*)\setminus \{0\}$, we get that $(k_1+k_2)= 0$, and as $k_1,k_2\in K(x^*)$, $k_1=-k_2$. It follows from the pointedness of the cone $K(x^*)$ that $k_1=k_2=0$, contradicting the fact that $k_1\neq 0$.
\end{proof}

It is worth mentioning that the concept of $K$--convexity for $F$ depends of the point-to-cone mapping $K$. Thus, this general approach covers several convexity concepts, from the scalar setting to the vector one and it can be used to model a large number of applications; see, for instance, \cite{bao-mord-soub, bao-mord-soub2, Eichfelder-med}.  In Section $5$ we discuss another variable order when the point-to-cone map depends of the image set of $F$, such kind of variable orders was introduced and studied in \cite{luis-gema-yunier-2014,nous}.

The Inexact Projected Gradient Method to solve problem \eqref{(P)} is presented in the next section.

\section{The Inexact   Projected Gradient Method}

This section is devoted to present an inexact projected gradient method for solving constrained smooth problems equipped with a variable order. This method uses an Armijo-type line-search, which is done on inexact descent feasible directions. The proposed scheme here has two main differences with respect to the approach introduced in \cite{nous}. {(i) it solves constrained problems.  (ii) it accepts approximate  {directions} with some tolerance.}

In the following, several constrained concepts and results will be presented and proved, which will be used in the convergence analysis of the proposed method below.

We start this section by presenting some definitions and basic properties of some auxiliary functions and sets, which will be useful in the convergence analysis of the proposed algorithms.  
Firstly, we define the set valued mapping $G\colon\mathbb{R}^n\rightrightarrows \mathbb{R}^m$, which for each $x$, defines the set of the normalized generators of
$K^*(x)$, \emph{i.e.} $G(x)\subseteq K^*(x)\cap\mathbb{S}(0,1)$ is a compact set such that the cone generated by its convex hull is $K^*(x)$.
Although the set of the dual cone $K^*(x)\cap \mathbb{S}(0,1)$ fulfills those properties,  in general, it is possible to take smaller sets; see, for instance, \cite{jahn, jahn1, luc}. On the other hand, assuming that $F\in \mathcal{C}^1$, we consider {the support} function {$\rho\colon\mathbb{R}^n\times\mathbb{R}^m\to \mathbb{R}$ as
\begin{equation}\label{Drho}\rho(x,w):=\max_{y\in G(x)} y^T w.\end{equation} 
$\rho(x,w)$ was extensively studied for the vector optimization in  {Proposition} 3.1 of \cite{[20]} and it is useful to define the useful auxiliary function
$\phi\colon\mathbb{R}^n\times\mathbb{R}^n\to \mathbb{R}$, as
\begin{equation}\label{asconsec}\phi(x,v):=\max_{y\in G(x)} y^T J_F(x)v.\end{equation} 
Then, we are ready to introduce the following auxiliary subproblem,
for each $x\in \mathbb{R}^n$ and $\beta>0$, as
\begin{equation}\tag{$P_x$}\label{ee}
 \min_{v\in  C-x}\left\{ \frac{\| v\|^2}{2}+\beta \phi(x,v)\right\}.
\end{equation}}
\begin{remark}\label{newone}
Since $G(x)$ is compact, the function $\phi(x,\cdot)\colon\mathbb{R}^n\to \mathbb{R}$ is well defined for each $x\in \RR^n$. Moreover, it is a continuous function.
\end{remark}
Next proposition provides a characterization of the stationarity using the auxiliary function $\phi$, defined in \eqref{asconsec}. The unconstrained version of the following proposition can be found in Proposition 4.1 of \cite{nous}. {A version of this proposition was presented in Lemma 2.4 of \cite{[19]} for vector optimization}. 

\begin{proposition}\label{prop1}The following statements hold:
 \item[ {\bf(i)}] For each $x\in \mathbb{R}^n$, $\max_{y\in G(x)} y^T\hat{w}<0$ if and only if $\hat{w}\in -{\rm int}( K(x))$.
 \item[ {\bf(ii)}]The point $x$ is not   stationary   iff there exists $v\in C-x$ such that $\phi(x,v)<0$.
 \item[ {\bf(iii)}]If $\phi(x,v)<0$ and {$\beta>0$}, then   there exists $\overline{\lambda}>0 $ such that  $\displaystyle \frac{\|\lambda v\|^2}{2}+\beta \phi(x,\lambda v)<0 $ for all $\lambda\in (0, \overline{\lambda} ]$.
 \item[ {\bf(iv)}] For each $x\in \mathbb{R}^n$,  subproblem \eqref{ee} has a unique solution, denoted by $v(x)$.
\end{proposition}
\begin{proof}
{\bf (i):} The result of this item follows as in Proposition 4.1(i) of \cite{nous}. 

\noindent {\bf (ii):} Note that, fixing $x$, it follows from \eqref{asconsec} that $\phi(x,v)= \rho(x, J_F(x)v)$. Then,  by the definition of stationarity  and item (i), the statement holds true.

\noindent {\bf (iii):}  It follows from the definition of $\phi(x,v)$ that $\phi(x,\cdot)$ is a positive homogeneous  function. Thus, for all $\lambda>0$,  
\begin{equation}
\label{abEq}
\frac{\|\lambda v\|^2}{2}+\beta \phi(x,\lambda v)=\lambda\left(\lambda  \frac{\|  v\|^2}{2}+\beta  \phi(x,  v)\right).\end{equation}
Since   $\phi(x,  v)<0$, there exists $\bar \lambda>0$ small enough such that $\bar \lambda  \displaystyle\frac{\|  v\|^2}{2}+\beta  \phi(x,  v)<0.$ Hence, \eqref{abEq} together with the above inequality implies that $\displaystyle\frac{\|\lambda v\|^2}{2}+\beta \phi(x,\lambda v)<0,$ for all $\lambda\in (0,\bar \lambda]$, as desired.
\medskip

\noindent {\bf (iv):} Using the definition of the function $\phi(x,v)$, given in \eqref{abEq}, it is easy to prove that  {$\phi(x,\cdot)$} is a sublinear function as well.  
Hence, $\phi(x,\cdot)$ is a convex function, and then, $\displaystyle\frac{\| v\|^2}{2}+\beta \phi(x,v)$ is a strongly convex function. Since $C$  is a convex set, $C-x$ is also convex and therefore,  subproblem \eqref{ee} has a unique minimizer.
\end{proof}

Based on Proposition \ref{prop1}(iii), we  {can} define   $v(x)$ as the unique solution of subproblem \eqref{ee} and  $y(x,v)$ is an element of the compact set $G(x)$
such that $y(x,v)^T J_F(x)v=\phi(x,v)$.  {Next} we will  discuss about the continuity {of the function \begin{equation}
\label{teta}
\theta_\beta(x):=\frac{\| v(x)\|^2}{2}+\beta \phi(x,v(x)),\end{equation} which is related with the one  defined in (35) of \cite{[21]}.} 

The following proposition is  {the constrained version} of Proposition 4.2 in \cite{nous}. {Items (i)-(ii), (iii) and (iv) can be seen as a version for the variable vector optimization of Proposition $3$ of \cite{[21]}, Proposition $2.5$ of \cite{[19]} and Proposition 3.4 of \cite{[20]}, respectively}.

\begin{proposition}\label{prop2} Let $F\in \mathcal{C}^1$ and fix $\beta>0$. Then, the following hold
 {\item[ {\bf(i)}]   $\theta_\beta(x)\leq 0$ for all $x\in C$.}
 {  \item[ {\bf(ii)}]   {$x\in C$} is a stationary point iff $\theta_\beta(x)=0$.}
 \item[ {\bf(iii)}] $\|v(x)\|\leq 2\beta\| J_F(x)\|$   {for all $x\in C$}. \item[ {\bf(iv)}]If $G$ is a closed map, then $\theta_\beta$ is an upper semi-continuous function on $C$.
\end{proposition}
\begin{proof}
{\noindent {\bf (i):} Note} that as $0\in C-x$ for all $x\in C$ and $\displaystyle\theta_\beta(x)\leq \frac{\|0\|^2}{2}+\beta\phi(x,0)=0.$

\noindent {{\bf (ii):}} As shown in Proposition \ref{prop1}(ii),
 $x$ is a non stationary point iff  for some $v\in C-x$, $\phi(x,v)<0$. Then, by Proposition \ref{prop1}(iii),  there exists  $\hat{v}\in C-x$ such that  $\displaystyle\frac{\lambda^2}{2}\|v\|^2+\lambda \beta\phi(x,v)< 0$ and hence  $\theta_\beta(x)<0$. 

\noindent {\bf (iii):} By (i), $0\geq \theta_\beta(x)=\displaystyle\frac{\|v(x)\|^2}{2}+\beta y(x,v(x))^T J_F(x)v(x).$ Then, after some algebra, we get
$$\frac{\|v(x)\|^2}{2}\leq -\beta y(x,v(x))^T J_F(x)v(x)\leq \beta\|  y(x,v(x))^T J_F(x)v(x) \|.$$
Using that $\| y(x,v(x))\|=1$, it follows from the above inequality that
$$\frac{\|v(x)\|^2}{2}\leq  \beta \| J_F(x)\|\|v(x)\|,$$ and the result follows {after dividing the above inequality by the} positive term  ${\|v(x)\|/2}\neq 0$.

\noindent {\bf (iv):} Now we prove the upper semi-continuity of the function $\theta_\beta$. Let $(x^k)_{k\in\NN}$ be a sequence converging to $x$. Take $\hat{x}\in C$ such that   $v(x)=\hat{x}-x$ and also denote $\hat x^k:=v^{k}+x^k$.
It is clear that, for all $k\in\NN$, $\hat{x}-x^k\in C-x^k$, and so,\begin{align}\label{yy}\nonumber \theta_\beta(x^k)&= \frac{\| \hat{x}^k-x^k\|^2}{2}+\beta \phi(x^k,\hat{x}^k-x^k)\\\nonumber&\le \frac{\| \hat{x}-x^k\|^2}{2}+\beta \phi(x^k,\hat{x}-x^k)\\&
=\frac{\| \hat{x}-x^k\|^2}{2}+\beta y_k^T J_F(x^k)(\hat{x}-x^k).\end{align}
Since each $y_k:=y(x^k,\hat x-x^k)$  belongs to the compact set $G(x^k)\subseteq K^*(x^k)\cap \mathbb{S}(0,1)\subseteq \mathbb{B}(0,1)$ for all $k\in\NN$, then the sequence $(y_k)_{k\in\NN}$ is bounded. Therefore, there exists a convergent subsequence of $(y_k)_{k\in\NN}$. We can assume  without lost of generality that $\lim_{k\to \infty} y_k= y$, and also since
$G$ is closed,    $y\in G(x)$. Taking limit in \eqref{yy}, we get
\begin{align*}\limsup_{k\to\infty} \theta_\beta(x^k)&\leq  \limsup_{k\to\infty}\displaystyle  \frac{\| \hat{x}-x^k\|^2}{2}\displaystyle+\beta y_k^T J_F(x^k)(\hat{x}-x^k)\\
&=\frac{\| \hat{x}-x\|^2}{2}+
\beta y^T J_F(x)(\hat{x}-x)\\&\leq
\frac{\| \hat{x}-x\|^2}{2}+\beta \phi(x,\hat{x}-x)=\theta_\beta(x).\end{align*}
Then, the function $\theta_\beta$, defined in \eqref{teta}, is upper semi-continuous.
\end{proof}
\begin{lemma}\label{remark1}
Consider any {$x, \hat x\in C$ and $z\in\RR^n$}. If  $ J_F$ is locally Lipschitz  {around $x$ for all $x\in C$}, $\mbox{d}_H(G(x),G(\hat x))\leq L_G\|x-\hat x\|$ for some $L_G>0$ and $C$ is bounded, then $${\left|\phi(x,z)-\phi(\hat x,z)\right|\leq L\|x-\hat x\|,}$$ for some $L>0$. Hence,  {for all $v \in \RR^n$, $\phi(\cdot,v)$} is a continuous function  {on $C$}. 
\end{lemma}
\begin{proof}By Proposition~4.1(iv) of \cite{nous}, {and using the Lipschitz assumption for $G$ in $C$, $\rho(x,w)$, defined in \eqref{Drho}, is also a Lipschitz function for all $(x,w)\in C\times W$} for any bounded subset  $W\subset \RR^n$.  That is\begin{equation}\label{lipsro}
 |\rho(x_1,w_1)-\rho(x_2,w_2)|\leq \hat{L}\|x_1-x_2\| +\|w_1-w_2\|,
\end{equation} 
 {for all $x_1,x_2\in C$ and $w_1,w_2\in W$, where $\|w_i\|\le M,$ $i=1,2$ with $M>0$ and $\hat{L}:=L_GM$}.
Taking \eqref{lipsro} for  $x_1=x$, $x_2=x^k$, $w_1= J_F(x)(\hat{x}^k-x^k)$ and
$w_2= J_F(x^k)(\hat{x}^k-x^k)$, {we get} \begin{align*}
\Big|\rho\big(x, J_F(x)(\hat{x}^k-x^k)\big)   -\rho\big(x^k, J_F(x^k)(\hat{x}^k-x^k)\big)\Big| &\leq   \hat{L}\|x-x^k\|+\|( J_F(x)- J_F(x^k))(\hat{x}^k-x^k)\|
\\
&\leq  \hat{L}\|x-x^k\|+\| J_F(x)- J_F(x^k)\|\|\hat{x}^k-x^k\|,
\end{align*}
because of {the compactness of}  $C$ and {the continuity of} $J_F$, $\| J_F(x)(\hat{x}^k-x^k)\|\leq M$ for all $k\in\NN$ and $x\in C$.
Noting that
$$\phi(x,\hat{x}^k-x^k)-\phi(x^k,\hat{x}^k-x^k)=
\rho\left(x, J_F(x)(\hat{x}^k-x)\right)-\rho\left(x^k, J_F(x^k)(\hat{x}^k-x^k)\right),
$$ and due to $ J_F$ is locally Lipschitz and \eqref{lipsro}, it follows that
  \begin{equation}\label{eq33}\left|\phi(x,\hat{x}^k-x^k)-\phi(x^k,\hat{x}^k-x^k)\right|\leq (\hat{L}+L_F\hat M)\|x-x^k\|,
      \end{equation}  {for all $x\in C$ with $\hat{L}:=L_GM$ and $L_F$ the Lipschitz constant of $J_F$ and $\hat M>0$ such $\|\hat{x}^k-x^k\|\le \hat M$ for all $k\in\NN$}. {This proves the continuity of $\phi$ in the first argument}.
\end{proof}

Now we can prove the lower semicontinuity of $\theta_\beta$ {by following similar ideas of the result presented in Proposition $3.4$ of \cite{[20]} for vector optimization.}
 \begin{proposition}\label{item(iv)} Let $F\in \mathcal{C}^1$ and consider any $x, \hat x\in C$ {with $C$ bounded}. Then, if $\mbox{d}_H(G(x),G(\hat x))\leq L_G\|x-\hat x\|$ for some $L_G>0$ and $ J_F$ is locally Lipschitz  {around $x$ for all $x\in C$},  $\theta_\beta$ is a lower semicontinuous function on $C$.
\end{proposition}
\begin{proof}
We consider the function $\theta_\beta(x)$. Note further that
\begin{align*}\label{eqm}
\theta_\beta(x)\leq& \beta\phi(x,\hat{x}^k-x)+\frac{\|\hat{x}^k-x\|^2}{2}\\
=&\theta_\beta(x^k)+\beta\left[\phi(x,\hat{x}^k-x)-\phi(x^k,\hat{x}^k-x^k)\right]+\frac{\|\hat{x}^k-x\|^2-\|\hat{x}^k-x^k\|^2}{2}\\
=&\theta_\beta(x^k)+\beta\left[\phi(x,\hat{x}^k-x)-\phi(x^k,\hat{x}^k-x^k)\right]+\frac{1}{2}\left[-2\langle \hat{x}^k, x^k-x\rangle +\|x\|^2-\|x^k\|^2\right].
\end{align*}
Thus, taking limit in
the previous inequality and using Lemma \ref{remark1}, we get
$$\lim_{k\to \infty}\phi(x,\hat{x}^k-x)-\phi(x^k,\hat{x}^k-x^k)=0.$$ Also, it is follows that $\lim_{k\to \infty}
	\frac{1}{2}\left[\|x\|^2-\|x^k\|^2\right]-\langle \hat{x}^k, x^k-x\rangle= 0.$ Hence,
\begin{align*} \theta_\beta(x)&\leq \liminf_{k\to\infty}\left\{
\theta_\beta(x^k)+\beta\left[\phi(x,\hat{x}^k-x)-\phi(x^k,\hat{x}^k-x^k)\right]-\langle \hat{x}^k, x^k-x\rangle +\frac{\|x\|^2-\|x^k\|^2}{2}\right\}\\
&= \liminf_{k\to\infty}  \theta_\beta(x^k),
\end{align*} establishing the desired result.
 \end{proof}
 
{Now we recall the concept of $\delta$-approximate direction introduced in Definition $3.1$ of \cite{[19]}.
\begin{definition}\label{deltasol} Let $x\in C$ and $\beta>0$. Given $\delta\in [0,1)$, we say that $v$ is a $\delta$-approximate solution of subproblem \eqref{ee} if $v\in C-x$ and
$\beta\phi(x,v)+\displaystyle\frac{\|v\|^2}{2}\leq (1-\delta) \theta_\beta(x).
$
If $v\neq 0$ we say that $v$ is a  $\delta$-approximate direction.
 \end{definition}}
Hence, from a numerical
point of view, it would be interesting to consider algorithms in which the line-search is {given over a} $\delta$-approximate {solution of subproblem \eqref{ee} instead of on an exact solution of it}.

\begin{remark} Note that if the solution of subproblem \eqref{ee} is $0$, then the only possible $\delta$-approximate solution is $v=0$. {In other case, since $\theta_\beta(x)<0$, there exist feasible directions $v$ such that} $$\beta\phi(x,v)+\frac{\|v\|^2}{2}\in\left[\theta_\beta(x), (1-\delta) \theta_\beta(x)\right].$$ In particular {$v(x)$}, the solution of subproblem \eqref{ee}, is always a $\delta$-approximate solution.\end{remark}
Next we present an inexact algorithm for solving problem \eqref{(P)}. {The algorithm requires
the following exogenous parameters: $ \delta\in[0,1)$ and $\sigma, {\gamma} \in(0,1)$ and $0<\bar \beta\le \hat{\beta}<+\infty$.}
\begin{center}\fbox{\begin{minipage}[b]{\textwidth}
\noindent{{\bf I}nexact {\bf P}rojected {\bf G}radient Method ({\bf IPG Method}).}\;\; {Assume that $\beta_k\in[\bar\beta,\hat{\beta}]$ for all $k\in\NN$}. 

\medskip

\noindent {\bf Initialization:}  Take $x^0\in \mathbb{R}^n$ and $\beta_0$.

\medskip

\noindent  {\bf Iterative step:} Given $x^k$ and $\beta_k$, compute $v^k$ a $\delta$-approximate solution of $(P_{x^k})$.
 If $v^k=0$, then stop. Otherwise compute
\begin{equation}\label{Armijo-type}
j(k):=\min\left\{j\in\NN \colon F(x^k)- F(x^k+\gamma^{j}v^k)+\sigma \gamma^{j} J_F(x^k)v^k\in K(x^k)\right\}.
\end{equation}
\noindent  Set
$
x^{k+1}=x^k+\gamma_kv^k\in C,
$
with $\gamma_k=\gamma^{j(k)}$.

\end{minipage}}\end{center}
{It is worth noting that {\bf IPG Method} extends Algorithm $3.3$ of \cite{[19]} to the variable order setting.}
Next proposition proves that the stepsize $\gamma_k$ is well defined for all $k\in \NN$, \emph{i.e.,} there exists a finite positive integer $j$ that fulfills Armijo-type rule given in \eqref{Armijo-type} at each step of {\bf IPG Method}. {The proof of the next result uses a similar idea to the presented in Proposition $2.2$ of \cite{[19]}.}
\begin{proposition} Subproblem \eqref{Armijo-type} has a finite solution, i.e.,  there exists an index $j(k)<+\infty$ which is solution of \eqref{Armijo-type}.
\end{proposition}
\begin{proof} If $v^k=0$ then {{\bf IPG Method} stops}. Otherwise, if  $v^k\neq 0$ then by Proposition \ref{prop2}(ii), $x^k$ is not a stationary point and $\theta_{\beta_k}(x^k)<0$. Moreover, $$\beta_k\phi(x^k,v^k)\leq \beta_k\phi(x^k,v^k)+\frac{\|v^k\|^2}{2}\le (1-\delta) \theta_{\beta_k}(x^k)<0.$$
Note further that  $\phi(x^k,v^k)=\max_{y\in G(x^k)} y^T J_F(x^k)v^k<0.$ Thus, it follows from Proposition \ref{prop1}(i) that
\begin{equation}\label{tt}
  J_F(x^k)v^k\in -{\rm int}(K(x^k)).
\end{equation}
Using the Taylor expansion of $F$ at $x^k$, we obtain that
\begin{equation}\label{igual-para-armijo}F(x^k)+\sigma \gamma^{j} J_F(x^k)v^k- F(x^k+\gamma^{j}v^k)=(\sigma -1)\gamma^{j} J_F(x^k)v^k +o(\gamma^{j}).\end{equation}
Since $\sigma<1$ and $K(x^k)$ is a cone,  it follows from \eqref{tt} that  $(\sigma -1)\gamma^{j} J_F(x^k)v^k\in {\rm int}(K(x^k)).$
Then, there exists $\ell\in\NN$ such that, for all
$j\ge \ell$, we get $(\sigma -1)\gamma^{j} J_F(x^k)v^k +o(\gamma^{j})\in K(x^k).$ Combining the last inclusion with \eqref{igual-para-armijo}, we obtain
$F(x^k)+\sigma \gamma^{j} J_F(x^k)v^k- F(x^k+\gamma^{j}v^k)\in K(x^k)$ for all $j\ge \ell$. {Hence \eqref{Armijo-type} holds for $j(k)=\ell$}. \end{proof}
\begin{remark}{\rm After this {proposition} it is clear that given $(x^k,v^k)$, $j(k)$ is well-defined. Furthermore, the sequence generated by {\bf IPG Method} is always feasible. Indeed, as $x^k, x^k+v^k\in C$, $\gamma_k\in (0,1]$ and  $C$ is convex,   $x^{k+1}=x^k+\gamma_kv^k\in C$.} 
\end{remark}
\section{Convergence Analysis of IPG Method}

In this section we prove the convergence of {\bf IPG Method} presented in the previous section. First we consider the general case and then
the result is refined for $K$--convex functions. From now on, $(x^k)_{k\in\NN}$ denote the sequence generated by {\bf IPG Method}.
We begin the section with the following lemma.

\begin{lemma}\label{le} Let $F\in \mathcal{C}^1$.
Assume that $\cup_{x\in C} K(x)\subseteq \mathcal{K}$, where
$\mathcal{K}$ is a closed, pointed and convex cone. 
If   $x^*$ is an accumulation point of $(x^k)_{k\in\NN}$, then $\lim_{k\to\infty}F(x^k)= F(x^*)$.
\end{lemma}
 \begin{proof} Let $x^*$ be any accumulation point of the sequence $(x^k)_{k\in\NN}$ and denote $(x^{i_k})_{k\in\NN}$ a subsequence of $(x^k)_{k\in\NN}$ such that  $\lim_{k\to\infty}x^{i_k}= x^*$. It follows from the definition of Armijo-type line-search in \eqref{Armijo-type} that
\begin{equation}\label{fg}F(x^{k+1}) -F(x^k)-\sigma\gamma_k J_F(x^k)v^k\in -K(x^k).\end{equation}
Since {\bf IPG Method} does not stop after finitely many steps, $v_k\neq 0$, which means that {$\phi(x^k,v^k)<0$}. By  Proposition \ref{prop1}(i), this means that $ J_F(x^k)v^k\in -{\rm int}(K(x^k)).$
Multiplying the last inclusion by $\sigma\gamma_k>0$ and summing with \eqref{fg}, we get from the convexity of $K(x^k)$ that 
$$F(x^{k+1}) -F(x^k)-\sigma\gamma_k J_F(x^k)v^k+\sigma\gamma_k  J_F(x^k)v^k\in -{\rm int}(K(x^k)). $$ 
Thus, 
$  F(x^{k+1})-F(x^{k})\in - {\rm int}(K(x^k)).$
Since $\cup_{x\in C} K(x)\subseteq \mathcal{K}$ {from assumption}, it holds that ${\rm int}(K(x))\subseteq {\rm int}(\mathcal{K})$  for all $x$, and
$ F(x^{k+1})-F(x^{k})\in -{\rm int}( \mathcal{K}).$ {Hence, $(F(x^k))_{k\in\NN}$ is decreasing with respect to cone $\mathcal{K}$.
The continuity of $F$ imply that $\lim_{k\to \infty} F(x^{i_k})= F(x^*).$ 
Then, to prove that the whole sequence $(F(x^k))_{k\in\NN}$ converges to  $F(x^*)$, we use  {the fact that the whole sequence $(F(x^k))_{k\in\NN}$ is} decreasing with respect to cone $\mathcal{K}$, which is a closed, pointed and convex cone;
see, for instance,  {Proposition 3.1, pages 90, 91 of \cite{Peressini} and Example 23.4 of \cite{danielltammer}}. Thus, we get that
$\lim_{k\to \infty} F(x^k)= F(x^*),$ as desired.}\end{proof}

We present an analogous result as was proved in Proposition \ref{prop2}(iii) where $v^k$ is a $\delta$-solution of subproblem $\left(P_{x^k}\right)$, which gives us a upper bound for the norm of $v^k$. {Next lemma is a version of Proposition 2.5 of \cite{[19]} to the variable order setting.}
\begin{lemma}\label{opu} {Let $(x^k)_{k\in\NN}$ and $(\beta_k)_{k\in\NN}$ be sequences generated by {\bf IPG Method}. Then, $\|v^k\|\leq 2 \beta_k\| J_F(x^k)\|.$}\end{lemma}\begin{proof} By  the definition of $\delta$-approximate direction
$\beta_k\phi(x^k,v^k)+\displaystyle\frac{\|v^k\|^2}{2}\leq (1-\delta) \theta_{\beta_k}(x^k).$
As was shown in Proposition \ref{prop2}(ii), $(1-\delta) \theta_{\beta_k}(x^k)\leq 0$, since $x^k\in C$. Thus, $\displaystyle\frac{\|v^k\|^2}{2}\leq -\beta_k\phi(x^k,v^k)$ and the result follows as in Proposition \ref{prop2}(iii).\end{proof}

Next we prove the stationarity of the accumulation points of the generate sequence. {Some arguments used in the proof of the next theorem are similar to those of Theorem $3.5$ of \cite{[19]} and Theorem $5.1$ of \cite{nous} for {fixed} and variable vector optimization, respectively.}
\begin{theorem}\label{t1}Suppose that
\item[ {\bf(a)}] $\cup_{x\in C} K(x)\subseteq \mathcal{K}$, where
$\mathcal{K}$ is a  closed, pointed and convex cone. \item[ {\bf(b)}] {The map $G$ is  closed}. \item[ {\bf(c)}] $\mbox{d}_H(G(x),G(\hat{x}))\leq L_G\|x-\hat{x}\|$, for all $x,\hat{x}\in C$.
\item[ {\bf(d)}] {$J_F$} is a locally Lipschitz function  {around $x$, for all $x\in C$}.
\item [] {If $(\beta_k)_{k\in\NN}$ is a bounded sequence}, then all {the} accumulation points of $(x^k)_{k\in\NN}$  are stationary points of problem \eqref{(P)}.
\end{theorem}
\begin{proof} Let $x^*$ be an accumulation point of the sequence $(x^k)_{k\in\NN}$. Denote $(x^{i_k})_{k\in\NN}$ any convergent subsequence to $x^*$.  Since $F\in \mathcal{C}^1$, Lemma \ref{opu} implies that the subsequence $(v^{i_k})_{k\in\NN}$ is also bounded and hence has a convergent subsequence. Without loss of generality, we assume that  $(v^{i_k})_{k\in\NN}$ converges to $v^*$, {$\beta_{i_k}$ and $\gamma_{i_k}$ converge to $\beta_*\ge \bar \beta$ and $\gamma^*$, respectively}.  {Recall that}  $\rho(x,w)=\max_{y\in G(x)} y^Tw$. 

By definition we have
$F(x^{k+1}) -F(x^k)-\sigma\gamma_k J_F(x^k)v^k\in -\mathcal{K}.$
Using  Proposition \ref{prop1}(i), implies that $\rho\left(x^{i_k},F(x^{k+1}) -F(x^k)-\sigma\gamma_k {J_F(x^k)v^k}\right)\leq 0$. Since the function
 $\rho$ is sublinear, as shown in
Proposition \ref{prop1} (iv), we get
\begin{equation}\label{bolita}
\rho\left(x^k, F(x^{k+1})-F(x^k)\right)\leq \sigma\gamma_k\rho\left(x^k, J_F(x^k)v^k\right) .
\end{equation}
{Using the semilinear property for $\rho$ in the second argument, we can rewrite}  \eqref{bolita} as
$\rho\left( x^k,F(x^k)\right) - \rho\left(x^k,F(x^{k+1})\right)\geq -\sigma\gamma_k\rho\left(x^k, J_F(x^k)v^{k}\right)\geq 0$. {Now 
considering the subsequences $(x^{i_k})_{k\in\NN}$  and $(v^{i_k})_{k\in\NN}$, where {$v^{i_k}=v(x^{i_k})$} on this inequality, we have}
$$\lim_{k\to\infty}  {\left[\rho\left(x^{i_k},F(x^{i_k})\right) - \rho\left(x^{i_k},F(x^{i_{k}+1})\right)\right]}\geq -\sigma \lim_{k\to\infty} \gamma_{i_k} \rho\left(x^{i_k},
 J_F(x^{i_k}) v^{i_k}\right)\geq 0.$$As already was observed in {the proof of} Lemma \ref{remark1},  {from  {\bf (c)} and  {\bf (d)} we have that}
$\rho$ is continuous and moreover from Lemma \ref{le}, we have $\lim_{k\to\infty} F(x^k)=F(x^*)$. Thus,
$$\lim_{k\to\infty}  {\left[ \rho\left(x^{i_k},F(x^{i_k})\right) - \rho\left(x^{i_k},F(x^{i_{k}+1})\right)\right]}=\rho\left(x^*,F(x^*)\right)-
\rho\left(x^*,F(x^*)\right)=0.$$
These facts imply that
$\lim_{k\to\infty} \gamma_{i_k} \rho\left(x^{i_k},  J_F(x^{i_k})v^{i_k}\right)=0.$
Hence we can split our analysis in two cases $\gamma^*> 0$ and $\gamma^*=0$.
\medskip 

\noindent {\bf Case 1:} $ \gamma^*> 0$.
 Here 
\begin{equation}\label{rrr}\lim_{k\to\infty} \phi(x^{i_k}, v^{i_k})=\lim_{k\to\infty} \rho\left(x^{i_k},  J_F(x^{i_k})v^{i_k}\right)=0.\end{equation}
Suppose that \begin{equation}\label{tttt}\theta_{\beta_*}(x^*)=\|v(x^*)\|^2/2+\beta_*\phi(x^*,v(x^*)) <-\epsilon<0, \end{equation} where $v(x^*)=\hat{x}-x^*$ {with $\hat x\in C$}. Due to the continuity of $\phi(\cdot,\cdot)$ in both arguments, Lemma \ref{remark1} and \eqref{rrr} imply that $$\phi(x^{i_k},v^{i_k})>-\frac{(1-\delta)\epsilon}{\max_{k\in\NN} \beta_k}=-\frac{(1-\delta)\epsilon}{\hat \beta}$$ for $k$ large enough. After note that $(\beta_k)_{k\in\NN}$ is a positive and bounded sequence, then
\begin{equation}\label{jj}
\|v^{i_k}\|^2/2+\beta_{i_k}\phi(x^{i_k},v^{i_k})\geq \beta_{i_k}\phi(x^{i_k},v^{i_k})>-\beta_{i_k}\frac{(1-\delta) \epsilon}{\hat\beta}\geq-(1-\delta) \epsilon.
\end{equation}
By definition of the subsequence $(v^{i_k})_{k\in\NN}$, we have,  for all $v^{i_k}\in C-x^{i_k}$ {and $v\in C-x^{i_k}$},
\begin{equation}\label{j}
(1-\delta)\left( \frac{\|v \|^2}{2}+\beta_{i_k} \phi(x^{i_k},v)\right)\geq (1-\delta)
\theta_{\beta_{i_k}}(x^{i_k})\geq 
 \frac{\|v^{i_k}\|^2}{2}+\beta_{i_k} \phi(x^{i_k},v^{i_k}).
\end{equation}
Combining \eqref{jj} and \eqref{j}, we obtain that $(1-\delta)\left(\displaystyle\frac{\|v\|^2}{2}+\beta_{i_k} \phi(x^{i_k},v)\right) >- (1-\delta)\epsilon$. 
In particular consider $\hat{v}^k=\hat{x} -x^{i_k}$. Dividing by $(1-\delta)>0$, we obtain 
$$  \frac{\|\hat{v}^k\|^2}{2}+\beta_{i_k} \phi(x^{i_k},\hat{v}^k) >- \epsilon . $$
By the continuity of function $\phi$ with respect to the first argument and taking limit in the previous inequality, lead us to the following inequality $\|v(x^*)\|^2/2+{\beta^*\phi(x^*,v(x^*))}\ge- \epsilon$. This fact and \eqref{tttt} imply
$$-\epsilon>\frac{\|v(x^*)\|^2}{2}+\beta_*  \phi({x^*},v(x^*))  \geq - \epsilon, $$ which is a contradiction.
Thus, we can conclude that  $\theta_{\beta_*}(x^*)\geq 0$ and, hence, using {Proposition \ref{prop2}},  $x^*$ is a stationary point if $\limsup_{k\to\infty}  \gamma_{i_k}> 0$.

\noindent {\bf Case 2:} $\gamma^*=0$.
We consider the previously defined convergent subsequences $(x^{i_k})_{k\in\NN} $, $(\beta_{i_k})_{k\in\NN}$, $(v^{i_k})_{k\in\NN} $, $(\gamma_{i_k})_{k\in\NN}$ convergent to $x^*$, $\beta_*$, $v^*$ and $\gamma^*=0$, respectively.  Since $\beta_*>0$,  we get that
$$\rho\left(x^{i_k}, J_F(x^{i_k})v^{i_k}\right) \leq \rho\left(x^{i_k}, J_F(x^{i_k})v^{i_k}\right) +\frac{\|v^{i_k}\|^2}{2\beta_{i_k}}.$$
Since $v^{i_k}$ is a $\delta$-approximate solution of {$(P_{x^{i_k}})$}, see Definition \ref{deltasol}, then  
$$  \rho\left(x^{i_k}, J_F(x^{i_k})v^{i_k}\right) +\frac{\|v^{i_k}\|^2}{2\beta_{i_k}}\leq \frac{(1-\delta)}{\beta_{i_k}}\theta_{\beta_{i_k}}(x^{i_k})<0.$$
 {Recalling that $C$ is closed and {\bf (b)} and {\bf (c)} hold by Propositions \ref{prop2} and \ref{item(iv)}, we have that $\theta$ is a continuous function. Taking limits above, we get} $\rho(x^*, J_F(x^*)v^*)\leq  - \displaystyle\frac{\|v^*\|^2}{2\beta_*}\leq 0.$
Fix $q\in \mathbb{N}$. Then, for $k$ large enough
$F(x^{i_k}+\gamma^{q}v^{i_k})\notin F(x^{i_k})+\sigma\gamma^q  J_F(x^{i_k})v^{i_k}- K(x^{i_k}),$
as there exists $\hat{y}_{i_k}\in G(x^{i_k})$ such that
$\left\langle  F(x^{i_k}+\gamma^{q}v^{i_k}) - F(x^{i_k}) -\sigma\gamma^q  J_F(x^{i_k})v^{i_k},  \hat{y}_{i_k}\right\rangle  >0,$
it holds that
$$\rho\left(x^{i_k},  F(x^{i_k}+\gamma^{q}v^{i_k})-F(x^{i_k})- \sigma\gamma^q  J_F(x^{i_k})v^{i_k}\right)\geq 0.$$
Taking limit as $k$ tends to $+\infty$,  and using that $\rho$ is a continuous function, then
$$\rho\left(x^*,F(x^*+\gamma^{q}v^*)-F(x^*)-\sigma\gamma^q  J_F(x^*)v^*\right)\geq 0.$$
But {$\rho(x,\cdot)$ is a positive homogeneous} function, so,

$$\rho\left(x^*,\frac{F(x^*+\gamma^{q}v^*)-F(x^*)}{\gamma^{q}}- \sigma  J_F(x^*) v^*\right) \geq 0.$$
Taking limit as $q$ tends to $+\infty$, we obtain $\rho\left(x^*,(1-\sigma)  J_F(x^*) v^*\right)\geq 0.$
Finally, since
$\rho\left(x^*,  J_F(x^*) v^*\right)\leq 0$, it holds
$\rho\left(x^*,  J_F(x^*) v^*\right)= 0$
and by Proposition \ref{prop1}(ii), this is equivalent to say that $x^*\in S^s.$
\end{proof}

The above result generalizes  Theorem $5.1$ of \cite{nous}, where the exact steepest descent method for  unconstrained problems was studied.  Recall that at the exact variant of the algorithm the direction $v^k$ is computed as an exact solution of problem  $(P_{x^k})$.  In order to fill the gap  between these two cases, we present two direct consequences  of the above result, the inexact method for unconstrained problems and the exact method for the constrained problem.

\begin{corollary}Suppose that   conditions (a)-(d) of Theorem \ref{t1} are fulfilled.  
Then all accumulation points of the sequence $(x^k)_{k\in\NN}$ generated by the exact variant of {\bf IPG Method} are stationary points of problem \eqref{(P)}.
\end{corollary}
\begin{proof} Apply Theorem \ref{t1} to the case {$\delta=0$}.\end{proof} 

\begin{corollary}
Suppose that conditions (a)-(d) of Theorem \ref{t1} are fulfilled for $C=\RR^n$.  
If $(\beta_k)_{k\in\NN}$ is a bounded sequence, then all accumulation points of $(x^k)_{k\in\NN}$  computed by {\bf IPG Method} are stationary points of problem \eqref{(P)}.
\end{corollary}
\begin{proof} {Directly by applying Theorem \ref{t1} for $C=\RR^n$}. \end{proof}

The result presented in Theorem~\ref{t1} assumes the existence of accumulation points. We want to emphasize that this is a fact that takes place even when the projected gradient method is applied to the solution of classical scalar problems, i.e., $m=1$ and $K(x)=\RR_+$. The convergence of the whole sequence generated by the algorithm is only possible under stronger assumptions as convexity.
Now, based on quasi-F\'ejer theory,  we will prove the full convergence of the sequence generated by {\bf IPG Method} when we assume that $F$ is $K$--convex. We start by presenting its {definitions} and its
properties.
\begin{definition}\label{def-cuasi-fejer}
Let $S$ be a nonempty subset of $\mathbb{R}^n$. A
sequence $(z^k)_{k\in\NN}$ is said to be quasi-Fej\'er convergent
to $S$ iff for all $x \in S$, there exists $\bar{k}$ and a summable sequence
$(\varepsilon_k)_{k\in\NN}\subset \mathbb{R}_+$ such that $\| z^{k+1}-x\|^2 \leq \|
z^{k}-x\|^2+\varepsilon_k$ for all $k\ge\bar{k}$.
\end{definition}

This definition originates in \cite{browder} and has been further
elaborated in \cite{IST}. A useful result on quasi-Fej\'er sequences is the following.
\begin{fact}\label{cuasi-Fejer}
If $(z^k)_{k\in\NN}$ is quasi-Fej\'er convergent to $S$ then,
 \item[ {\bf(i)}] The sequence $(z^k)_{k\in\NN}$ is bounded.
 \item[ {\bf(ii)}] If an accumulation point
of $(z^k)_{k\in\NN}$ belongs to $S$, then the whole sequence
$(z^k)_{k\in\NN}$ converges.
\end{fact}
\begin{proof}  {See  Theorem $1$ of \cite{bmauricio}.}\end{proof}

For guaranteeing the convergence of {\bf IPG Method}, we introduce the following definition {which is related with the one presented in Definition $4.2$ of \cite{[19]}.}
\begin{definition}\label{scompatible}Let $x\in C$. A direction $v\in C - x$ is scalarization compatible (or
simply s-compatible) at $x$ if there exists $w\in {\rm conv}(G(x))$ such that
$v = P_{C-x}(-\beta J_F(x)w).$\end{definition}

In the following we present the relation between inexact and s-compatible directions.

\begin{proposition} Let $x\in C$, $w\in {\rm conv}(G(x))$, $v = P_{C-x}( -\beta J_F(x)w)$ and $\delta\in [0, 1)$.
If
$$\beta\phi( J_F(x)v)\leq (1-\delta)\beta \langle w,  J_F(x)v\rangle - \frac{\delta}{2}\|v\|^2,$$
then $v$ is a $\delta$-approximate projected gradient direction.\end{proposition}
\begin{proof} See Proposition 4.3 of \cite{[19]}. \end{proof}

We start the analysis with a technical result {which is related with the proof of Lemma $5.3$ of \cite{[19]}.} 
\begin{lemma}\label{auxi1} Suppose that $F$ is $K$--convex. Let $(x^k)_{k\in\NN}$ be a sequence generated by
{\bf IPG Method} where $v^k$ is an $s$-compatible direction at $x^k$, given by
{$v^k =  P_{C-x^k}(-\beta_k J_F( {x^k})w^k)$}, with $w^k\in {\rm conv}(G(x^k))$ for all $k\in\NN$. If for a given $\hat{x}\in C$ we have
$F(\hat{ x})-F(x^k)\in -K(x^k)$, then
$$\|x^{k+1}- \hat{x}\|^2 \leq \|x^k- \hat{x}\|^2+
 2\beta_k\gamma_k|\langle w^k, J_F(x^k)v^k\rangle|.$$\end{lemma}
\begin{proof} Since $x^{k+1} = x^k + \gamma_kv^k$, we have
$\|x^{k+1}- \hat{x}\|^2 =\|x^k- \hat{x}\|^2+\gamma_k^2\|v^k\|^2-2\gamma_k\langle v^k,\hat{x}-x^k\rangle.$
Let us analyze the rightmost term of the above expression. It follows from the definition of $v^k$
and the obtuse angle property of projections that
$\langle-\beta_k J_F(x^k)w^k-v^k, v-v^k\rangle\leq  0,$ for all $v\in C- x^k$.
Taking $v = \hat{x}- x^k\in C-x^k$ on the above inequality, we obtain
$$-\langle v^k, \hat{x}-x^k\rangle \leq  \beta_k\langle w^k, J_F(x^k)( \hat{x}-x^k)\rangle- \beta_k\langle w^k,   J_F(x^k)v^k\rangle-\|v^k\|^2.$$
Now, it follows from the convexity of $F$ that $\langle w^k, J_F(x^k)(\hat{x}-x^k)\rangle \leq  \langle w^k,F(\hat{x})-F(x^k)\rangle$. Also the fact {$F(\hat{x})\preceq_{K(x^k)}  F(x^k)$, i.e., $F(\hat{x})-F(x^k)\in -K(x^k)$,} together with $w^k\in K^*(x^k)$ imply that $\langle w^k,F(\hat{x})-F(x^k)\rangle \leq 0$. Moreover, {by using $J_F(x^k)v^k  \in {\rm int}(-K(x^k))$ and $w^k$ in $ {\rm conv}(G(x^k))=K^*(x^k)$}, we have 
$\langle w^k,  J_F(x^k)v^k\rangle < 0$ . Thus,  we get 
$-\langle v^k,\hat{x}-x^k\rangle \leq \beta_k|\langle w^k,  J_F(x^k)v^k\rangle|-\|v^k\|^2.$
The result follows because  $\gamma_k\in (0,1]$.\end{proof}

We still need to make a couple of supplementary assumptions, which are standard
in convergence analysis of classical (scalar-valued) methods and in its extensions to the vector
optimization setting.

\noindent {\bf Assumption 4.4:} Let  $(z^k)_{k\in\NN}\in F(C)$ be a sequence  such that   $z^k-  z^{k+1}\in K(x^k)$ for all $k\in\NN$
and $z\in F(C)$, $z^k-z\in \mathcal{K}$ for some closed, convex and pointed cone $\mathcal{K}$, $\cup_{k\in\NN} K(x^k)\subset \mathcal{K}$. Then there exists $\hat{x}\in C$ such that {$F(\hat{x})\preceq_\mathcal{K} z^k$ for all $k\in\NN$, i.e., $F(\hat{x})-z^k\in-\mathcal{K}$}. {In this case, we said that the sequence $(z^k)_{k\in\NN}$ is bounded below with respect to $\mathcal{K}$.}

Recently, it was observed in \cite{Kim-2016} that this assumption could be replaced by assuming that the restriction of $F$ on $C$ has compact sections. This assumption is related to the completeness of the image of $F$. It is important to mention that completeness is a standard assumption for ensuring existence of efficient points in vector problems in \cite{luc}.

\noindent {\bf Assumption 4.5:} The search direction $v^k$ is s-compatible at $x^k$, that is to say, $v^k =P_{C-x^k} (-\beta J_F(x^k)^Tw^k)$, where $w^k\in {\rm conv}(G(x^k))$ for all $k\in\NN$.

This assumption holds automatically in the exact case. Moreover, it has been widely used in the literature in the vector case; see, for instance, \cite{[19]}. {Version of these  assumptions are also used in \cite{[19]}, when the order is given by a constant cone.} 

{The} next result is an extension to the variable order setting of Theorem  {5.6} of \cite{[19]}.

\begin{theorem}\label{teo4.3g} Assume that $F$ is $K$--convex and that Assumptions 4.4 and 4.5 hold. If ${\rm int}(\cap_{k\in\NN} K(x^k))\neq \varnothing$  and there exists $\mathcal{K}$, a pointed, closed and convex cone such that {$K(x^k)\subset \mathcal{K}$ for all $k\in\NN$},
then every sequence generated by the inexact projected gradient method {\rm (}{\bf IPG Method}{\rm)} is bounded and its accumulation points are 
 weakly efficient solutions.\end{theorem}
\begin{proof} Let us consider the set $T :=\{x\in C: F(x^k)- F(x)\in  K(x^k), \text{ for all }k\}$,
and take $\hat{x}\in T$, which exists  {as consequence of Assumption 4.4 together with \eqref{Armijo-type} and the fact that $\cup_{k\in\NN} K(x^k)\subset \mathcal{K}$}.  Since $F$ is a  $K$--convex function and {Assumption
4.5} holds, it follows from Lemma \ref{auxi1} that
\begin{equation}\|x^{k+1}- \hat{x}\|^2\leq \|
x^k- \hat{x}\|^2 + 2\beta_k\gamma_k
|\langle
w^k, J_F(x^k)v^k\rangle |,\end{equation}
for all $k\in\NN$.  By the definition of $v^k$, it is a   descent condition. This means that 
$-  J_F(x^k)v^k\in K(x^k)$. Hence $\langle
w^k, J_F(x^k)v^k\rangle\leq 0.$
Then,
\begin{equation}\label{22}\|x^{k+1}- \hat{x}\|^2-\|
x^k- \hat{x}\|^2 \leq  2\beta_k\gamma_k
|\langle
w^k, J_F(x^k)v^k\rangle |\leq   - 2\beta_k\gamma_k
\langle
w^k, J_F(x^k)v^k\rangle.\end{equation}
On the other hand as  $\mathcal{K}$ is a closed, convex and  pointed cone with nonempty interior,
$\mathcal{K}^*$ is also a closed, convex and pointed cone with nonempty interior. Since $K(x^k)\subset\mathcal{K}$, it holds that $\mathcal{K}^*\subset K^*(x^k)$. Hence $\mathcal{K}^*\subset \cap_{k\in\NN} K^*(x^k)$.  Let $\omega_1,\ldots,\omega_m\in \mathcal{K}^*$ be a {basis} of $\RR^m$ {which exists because ${\rm int}(K^*)\neq \emptyset$}.

Then, there  {exist} $\alpha^k_1,\ldots, \alpha^k_m\in \RR$ such that  
$w^k=\sum_{i=1}^m \alpha^k_i\omega_i$. Substituting in \eqref{22},
\begin{equation}\|x^{k+1}- \hat{x}\|^2-\|
x^k- \hat{x}\|^2 \leq     - 2\beta_k\gamma_k\sum_{i=1}^m\alpha_i^k
\langle\omega_i\\
, J_F(x^k)v^k\rangle.\end{equation}
On the other hand, since   $- J_F(x^k)v^k\in K(x^k)$, $\omega_1,\ldots,\omega_m\in \mathcal{K}^*\subset K^*(x^k)$ and $\beta_k, \gamma_k> 0$ for all $k\in\NN$, it holds 
$\langle \omega_i, -2\beta_k\gamma_k  J_F(x^k)v^k\rangle \geq 0$. {Without loss of generality, we can assume that $\|\omega_i\|=1$  {because the normalized non-null vectors are still a basis of $\RR^n$}. Then,}  $\alpha_i^k$ is uniformly bounded, i.e. there   {exists} $M>0$ such that for all $k,i$
$|\alpha_i^k|\leq M.$
Hence,  
\begin{equation}\label{222}\|x^{k+1}- \hat{x}\|^2-\|
x^k- \hat{x}\|^2 \leq     - 2M\beta_k\gamma_k\sum_{i=1}^m
\langle\omega_i\\
, J_F(x^k)v^k\rangle.\end{equation}
By the Armijo-type line-search in \eqref{Armijo-type},
$F(x^{k+1})-F(x^k)-\gamma_k\sigma  J_F(x^k)v^k\in -K(x^k).$
Recall that $\omega_i\in \cap_{k\in\NN} K^*(x^k)$, we obtain 
$\displaystyle\frac{\langle \omega_i,  F(x^k)-F(x^{k+1})\rangle}{\sigma} \geq \langle \omega_i, -\gamma_k J_F(x^k)v^k\rangle.$
It follows from  \eqref{222} that
\begin{equation}\label{2221}\|x^{k+1}- \hat{x}\|^2-\|
x^k- \hat{x}\|^2 \leq     2\frac{M}{\sigma}\beta_k\sum_{i=1}^m
\langle\omega_i,   F(x^k)-F(x^{k+1})\rangle.\end{equation}
For the Fej\'er convergence of $(x^k)_{k\in\NN}$ to $T$,  it is enough to prove that the term $\beta_k\sum_{i=1}^m
\langle\omega_i,  F(x^k)-F(x^{k+1})\rangle\ge 0$ is summable at all $k\in\NN$. 
Since $\beta_k\le \hat\beta$ for all $k\in\NN$,
\begin{equation}
\label{des-20}
 \sum_{k=0}^{n}\beta_k\sum_{i=1}^m
\langle\omega_i,   F(x^k)-F(x^{k+1})\rangle \le \hat \beta \sum_{i=1}^m
\langle\omega_i,   F(x^0)-F(x^{n+1})\rangle.\end{equation}
As consequence of the Armijo-type line-search, we have $F(x^k)-F(x^{k+1})\in K(x^k)\subset \mathcal{K}$. So, $(F(x^k))_{k\in\NN}$ is a decreasing sequence with respect to $\mathcal{K}$. Furthermore {by {\bf Assumption 4.4}}, it is bounded below, also with respect to the order given by $\mathcal{K}$, by $F(\hat{x})$, where  $\hat{x}\in T$.
Hence,  {Proposition 3.1, pages 90, 91 of \cite{Peressini} implies that} the sequence $(F(x^k))_{k\in\NN}$ converges and using \eqref{des-20} in the inequality below, we get \begin{align*}\sum_{k=0}^{\infty}\beta_k\sum_{i=1}^m
\langle\omega_i,   F(x^k)-F(x^{k+1})\rangle&=\lim_{n\to\infty}\sum_{k=0}^{n}\beta_k\sum_{i=1}^m
\langle\omega_i,   F(x^k)-F(x^{k+1})\rangle\\&\le\hat\beta\lim_{n\to\infty}\sum_{i=1}^m
\langle\omega_i,   F(x^0)-F(x^{n+1})\rangle\\&= \hat\beta \sum_{i=1}^m
\langle\omega_i,   F(x^0)-\lim_{n\to \infty} F(x^{n+1})\rangle\\&= \hat\beta\sum_{i=1}^m
\langle\omega_i,   F(x^0)-F(\hat x)\rangle<+\infty.\end{align*} So, the quasi-Fej\'er {convergence}  is fulfilled. 

Since $\hat{x}$ is an arbitrary element of $T$, it is clear  that $(x^k)_{k\in\NN}$ converges quasi-Fej\'er to $T$. Hence, by  Fact \ref{cuasi-Fejer}, it follows that $(x^k)_{k\in\NN}$  is bounded. Therefore, $(x^k)_{k\in\NN}$
has at least one accumulation point, which, by Theorem \ref{t1} is stationary. By Proposition \ref{note1}, this point is also weakly efficient {solution}, because $F$ is $K$--convex. Moreover,
since $C$ is closed and the whole sequence is feasible, then this accumulation point belongs to
$C$. \end{proof}

\section{Another Variable Order}\label{sect6}

As was noticed in Section $6$ of \cite{luis-gema-yunier-2014} the variable order structure can be formulated in two different ways. Moreover, Examples 3.1 and 3.2 in \cite{luis-gema-yunier-2014} illustrate the differences of considering one order or the other. 
Thus, the variable order for the optimization problem may also depend of new order by using the  cone valued mapping
	$\hat K\colon\RR^m \rightrightarrows \RR^m$ where $\hat K(y)$ is a convex, closed and pointed cone for all $y\in {C}\subset \RR^m$ {where $C$ is the feasible set}. It is worth noting that the domain of the new mapping $\hat K$ is in $\RR^m$ and the orderings considered in the  previous sections, are defined by applications whose domain is $\RR^n$. As already discussed in \cite{luis-gema-yunier-2014}, convexity can be defined and convex  functions satisfy   nice properties such as the existence of subgradients. 

{Given a closed and convex set $C$, we say that $x^*\in C$  solves the optimization problem \begin{equation}\label{p234}\hat K-\min F(x) \text{ s.t. } x\in C,\end{equation}  if, for all $x\in C$,}
$$F(x)-F(x^*)\notin {-}\hat K(F(x^*))\setminus \{0\}.$$Here we can assume that 
$\hat K\colon F(C)\subseteq \RR^m \rightrightarrows \RR^m$. We shall mention that the main difference between the above problem and \eqref{(P)} yields in the definition of the variable order given now by $\hat K$.
For a more detailed study of the properties of the minimal points and their characterizations and convexity concept on this case; see \cite{gabrielle-book, luis-gema-yunier-2014}.

In this framework the definitions of weak solution and  stationary point are analogous. The main difference is that instead of $K(x^*)$, the cone $\hat K(F(x^*))$ is considered to define the variable partial order.  That is, the point $x^*$ is stationary, iff  for all $d\in  C-x^*$, we have
$ J_F(x^*)d\notin -{\rm int}(\hat K(F(x^*))).$
Then, similarly as in the case of problem \eqref{(P)}, the following holds.
\begin{proposition} If $F$ is a continuously differentiable function and $C$ is a convex set, weak solutions of problem \eqref{p234} are stationary points. Moreover if $F$ is also convex with respect to $\hat K$, the converse is true.\end{proposition}
\begin{proof} It follows the same lines of the proof of Propositions  \ref{primal} and \ref{note1}: the Taylor expansion of $F$ and the {closedness} of $\hat K(F(x^*))$ {imply} the result.  \end{proof}

The inexact algorithm is adapted in  the following way
 \begin{center}\fbox{\begin{minipage}[b]{\textwidth}
\noindent{{\bf F}-{\bf I}nexact {\bf P}rojected {\bf G}radient Method ({\bf FIPG Method}).} Given $0<\bar \beta\le \beta_k\le\hat{\beta}<+\infty$, $ \delta\in(0,1]$ and $\sigma, {\gamma} \in(0,1)$.  

\medskip

\noindent {\bf Initialization:}  Take $x^0\in \mathbb{R}^n$ and $\beta_0$.

\medskip

\noindent  {\bf Iterative step:} Given $x^k$ and $\beta_k$, compute $v^k$ a $\delta$-approximate solution of $(Q_{x^k})$.
 If $v^k=0$, then stop. Otherwise compute
\begin{equation}\label{Armijo1}
\ell(k):=\min\left\{\ell\in\NN \colon F(x^k)+\sigma \gamma^{\ell} J_F(x^k)v^k- F(x^k+\gamma^{\ell}v^k)\in \hat K(F(x^k))\right\}.
\end{equation}
\noindent  Set
$
x^{k+1}=x^k+\gamma_kv^k\in C,
$
with $\gamma_k=\gamma^{\ell(k)}$.
\end{minipage}}\end{center}
Here  the auxiliary problem
$(Q_{x^k})$ is defined as
 \begin{equation}\tag{$Q_{x^k}$}\label{iee}
 \min_{v\in  C-x^k}\left\{ \frac{\| v\|^2}{2}+\beta_k \phi(x^k,v)\right\},
\end{equation}  where 
 $\phi:\mathbb{R}^n\times\mathbb{R}^n\to \mathbb{R}$,
\begin{equation}\label{asconsecq}\phi(x,v):=\max_{y\in G(F(x))} y^T J_F(x)v,\end{equation} for $G:\mathbb{R}^m\rightrightarrows \mathbb{R}^m$  generator of ${\hat K}^*(F(x)):=\left[{\hat K}(F(x))\right]^*$.

With this  ordering{,} the function $\phi$ characterizes the stationarity. Furthermore, subproblem  $(Q_{x^k})$ has a unique solution  which is $v^k=0$ if and only if $x^k$ is a stationary point. 
Results analogous to those proven in Propositions \ref{prop2} and \ref{item(iv)} are also true. These facts implies that {\bf FIPG Method} is well defined, i.e., if it stops, then the computed point is a stationary point and in other case there exists $\ell(k)$ which satisfies the Armijo-type line-search \eqref{Armijo1}. So, only the convergence of a sequence generated by it must be studied. 

As in the last section, we analyze the convergence of the functional values sequence $(F(x^k))_{k\in\NN}$.

\begin{lemma}\label{lie}Suppose that $x^*$ is an accumulation point of $(x^k)_{k\in\NN}$ of the sequence generated by {\bf FIPG Method}.
If  $\cup_{x\in C} \hat K(F(x))\subseteq \mathcal{K}$, where
$\mathcal{K}$ is a closed, pointed and convex cone,   then $\lim_{k\to\infty}F(x^k)= F(x^*)$.
\end{lemma}
\begin{proof} The result is again proven  by the existence of a non-increasing sequence with an accumulation point. \end{proof}

Next, with the help of the last Lemma, {we summarize the convergence} of the generated sequence  with the following result.

\begin{theorem}\label{t11}Suppose that
\item[ {\bf(a)}] $\cup_{x\in C} \hat K(F(x))\subset \mathcal{K}$, where
$\mathcal{K}$ is a a closed, pointed and convex cone. \item[ {\bf(b)}] {The map $G\circ F$ is  closed}. \item[ {\bf(c)}] $\mbox{d}_H(G(F(x)),G(F(\hat{x})))\leq L_{GF}\|x-\hat{x}\|$, for all $x,\hat{x}\in C$.
\item[ {\bf(d)}] $J_F$ is a locally Lipschitz function  {around $x$ for all $x\in C$}.

\item [] {If $(\beta_k)_{k\in\NN}$ is a bounded sequence}, then all accumulation points of $(x^k)_{k\in\NN}$  generated by {\bf FIPG Method} are stationary points of problem \eqref{p234}.
\end{theorem}
\begin{proof} It follows from the same lines of the proof of Theorem \ref{t1}.  \end{proof}

We want to point out that in the last theorem the condition $$\mbox{d}_H(G(F(x)),G(F(\hat x)))\leq L_{GF}\|x-\hat{x}\|, \quad \forall x,\hat x\in C$$ {replaces the similar one given in Theorem \ref{t1}, i.e., 
$\mbox{d}_H(G(y),G(\hat{y}))\leq L_{G}\|y-\hat{y}\|,$  for all $y,\hat y\in C$. 
Moreover, if $F$ is Lipschitz on $C$, then this last condition implies the condition (c) in Theorem \ref{t11} by taking $y=F(x)$ and $\hat y=F(\hat x)$. }

{Next result is an extension to the variable order setting of Theorem 5.2 of \cite{[19]}}.
\begin{theorem} Assume that $F$ is $\hat K$--convex and additionally: \item[ {\bf (a)}] If $(z^k)_{k\in\NN}\subset F(C)$ is a sequence  such that   $z^k-  z^{k+1}\in \hat K(F(x^k))$ for all $k\in\NN$
and $z\in C$, $z^k-z\in \mathcal{K}$ for some closed, convex and pointed cone $\mathcal{K}$, $\cup_{k\in\NN} \hat K(F(x^k))\subseteq \mathcal{K}$, then there exists $\hat{x}\in C$ such that $F(\hat{x})\preceq z^k$ for all $k\in\NN$.
\item[ {\bf (b)}] The search direction $v^k$ is s-compatible at $x^k$, i.e., $v^k =P_{C-x^k} (-\beta J_F(x^k)^Tw^k)$, where $w^k\in {{\rm conv}(G(x^k))}$, for all $k\in\NN$. \item[ {\bf (c)}]${\rm int}(\cap_{k\in\NN} \hat K(F(x^k))) \neq \varnothing$.\item[ {\bf (d)}]There exists $\mathcal{K}$, a pointed, closed and convex cone such that $ \hat K(F(x^k))\subseteq \mathcal{K}$ for all $k\in\NN$.
  
\item [] Then every sequence generated by {\bf FIPG Method} is bounded and its accumulation points are 
 weakly efficient solutions.\end{theorem}
\begin{proof} It follows from the same lines of the proof of Theorem \ref{teo4.3g} using now the new variable order structure. \end{proof}

\section{Illustrative Examples}

In this section we present three examples, two for problem \eqref{(P)} and one for problem \eqref{p234}, illustrating how both proposed methods {work} {starting at ten different random initial points}. We verify our assumptions in each problem and make some comparisons between the proposed methods by changing the inexactness of the approximate gradient direction.

The algorithms were implemented in {\sc MatLab} R2012 and ran on an Intel(R) Atom(TM) CPU N270 at 1.6GHz. All starting points are not solutions and {are} randomly generated.  {The stopping criteria were $\|v^k\|<10^{-4}$ and also when a maximum of 30-iterations is reached. The solutions were displayed with four digits, CPU time was recorded in seconds and the number of iterations was also displayed in each case.}  {Our numerical test were performed with $\beta_k$ constant  equal $1$. }

{Despite the fact that it may not be an easy task to compute the positive dual cone of a given cone}, the computation of (approximate) directions is, in general, complicated. Indeed, after the definition, the exact optimal value of problem $(P_x)$ must be known. The use of $s$-compatible directions at iteration $k$ of the proposed methods, see Definition \ref{scompatible}, is recommended in the case in which the exact projection onto the set $C-x^k$ is not too complicated.
This is the case for the set of feasible solutions in the next example. Clearly, in all examples below, {the defining order cone-valued mappings are closed} maps and  {their generators are} Lipschitz with respect to the Hausdorff distance.
\begin{example}
We consider the vector problem as \eqref{(P)} with  $$K - \min\; F(x)=\left(x+1,x^2+1\right),\text{ s.t. } \; x\in [0,1],$$ where $F:\RR\rightarrow \RR^2$ and the variable order is given by  $K:\RR\rightrightarrows \RR^2$, $$K(x):=\left\{(z_1,z_2)\in \RR^2: z_1\geq 0,\;(x^2+1)z_1-(x +1)z_2\leq 0\right\}.$$
In this model 
the closed interval $[0, \sqrt{2}-1]\approx[0,\,0.4142]$ is the set of minimizers. 

{The {\bf IPG Method} was run ten times by using ten random initial points, outside of the set of minimizers, and each time it ended at solution points,} which  {have been} obtained after the verification of the stopping criterion. 

\noindent The method gives the following data:
{\small
$$\begin{array}{|l|c|c|c|c|c|c|c|c|c|c|}\hline {\bf {Instances}}& {\bf 1}&{\bf 2}&{\bf 3}&{\bf 4}&{\bf 5}&{\bf 6}&{\bf 7}&{\bf 8}&{\bf 9}&{\bf 10}\\\hline\hline  {\bf Initial\, Points} & 0.6557 &   0.6948&    0.8491  &  0.9340  &  0.6787&    0.7577 &   0.7431 &   0.4387  &  0.6555    &0.9502
\\\hline {\bf Solutions}&  0.4115 \,&   0.4128  \,&  0.4140  \,& 0.4135 \,&    0.4116 \,&   0.4131   \,& 0.4127 \,&   0.4136 \,&    0.4114\,&    0.4130 \,
\\\hline
 {\bf CPU \; Time} &0.0001& 0.0250&0.0001&  0.0156    &    0.0001 &        0.0001& 0.0001 &0.1094& 0.0156& 0.0781 \\\hline {\bf N\frac{o}{}\, Iterations}&16 &   19 &    23&    26&    17 &   20&    20&     4 &   16&    28
\\\hline\end{array}$$}\\
Note that in all cases  above optimal solutions were obtained. {Moreover,  to illustrate how the inexactness of the {\bf IPG Method} affects the performance (average of CPU time and the average of the number of iterations for $10$ instances) of the convergence is considering different values of the approximate solution parameter $\delta$.} \\
{
{\small
$$\begin{array}{|l|c|c|c|c|}\hline {\delta-{\bf Approximation \; Parameters}}\;& {\delta=0}&{\delta=0.25}&{\delta=0.5}&{\delta=0.75}\\\hline\hline  {\bf Average \; of \; CPU \; Time} &\; 0.4819&\; 0.2232&\;0.0244&\; 0.0251\\\hline {\bf Average \; of\;  N\frac{o}{}\, Iterations}&9 &   15 &    19 & 21
\\\hline\end{array}$$} \\
The above table shows that when the inexactness increases,  {the number of iterations increases. However, the CPU time decreases from $\delta=0$ to $\delta=0.5$, however in the last value of $\delta$ the CPU time slightly increases.} }
\end{example}
The next example is a non-convex problem corresponding to the model studied in the previous section.
\begin{example}{[\rm cf. Example $4.13$ of \cite{ap11}}] Consider the following vector problem as problem \eqref{p234} 
$$\hat{K} - \min\; F(x_1,x_2)=\left( x_1^2, x_2^2\right),\text{ s.t. }\pi\leq x_1^2+x_2^2\leq 2\pi,$$
where $F\colon \RR^2\rightarrow \RR^2$ and the variable order is given by the cone $\hat{K}\colon\RR^2\rightrightarrows \RR^2$, $$\hat{K}(y):=\left\{z=(z_1,z_2)\in \RR^2: \|z\|_2\leq\left[ \left(\begin{matrix} 2 &1\\-1& -1\end{matrix}\right)y \right]^T z/\pi\right\}.$$
The set of solutions (stationary points) of this problem is directly computed by using the definition of $S^{s}$ given in \eqref{stationary-inclusion} as $$\left\{(x_1,x_2)\in \RR^2: x_1^2+x_2^2=\pi\quad  \mbox{or}\quad  x_1^2+x_2^2=2\pi \right\} .$$
The {\bf FIPG Method} computes the following points: 
{\small
$$\begin{array}{|c|c|c|c|c|c|c|c|c|c|c|}\hline{\bf {Instances}}&{\bf 1}&{\bf 2}&{\bf 3}&{\bf 4}&{\bf 5}&{\bf 6}&{\bf 7}&{\bf 8}&{\bf 9}&{\bf 10}\\\hline\hline  {\bf Initial\, Points}& 1.8650   & 1.7525   & 2.4190&    1.9573&    0.7931&    1.2683&    1.8135& -2.0485  &-0.6446&   -0.8561\\
 &   1.6400  &  1.6350&    0.0835&    0.2813&   -2.0321&   -1.6814&    0.3050
&0.3229    &1.9606&    2.1011\\\hline
{\bf Solutions}&\! \!
  1.1632 \!\! &\! \!  0.9850 \!\! &\! \!   1.7705 \!\!&\! \!   1.7492\!\!&\!  \!  0.8634 \!\! &\! \! 1.4208 \!\! &\!  \! 1.7456\!\!&\!  \! -1.7438\!\!&\! \!   -0.8016\!\!&\!    \!-0.9535\! \!\\
 &   2.2204&    2.3050&    0.0841 &   0.2859&   -2.3532  & -2.0650&    0.3074& 0.3172&    2.3750&    2.3182
\\\hline
{\bf CPU\, Time}& 0.2969  &  0.2344  &  0.2188   & 0.1719   & 0.1563 &   0.5625 &   0.2656& 0.1875   & 0.2031&    0.3125 \\\hline\!\!{\bf  N\frac{o}{}\, Iterations}\!\!&2   &  4&     2&     2&     4&    17 &    3&     2&     2&     5\\\hline\end{array}$$}\\
Note that the solutions computed at all iterations of the proposed method belong to the set of optimal solutions.\end{example}
{ The last example from Section 9.2 of the book \cite{gabrielle-book} is widely studied.
\begin{example}{[\rm cf. Example $9.5$ of \cite{gabrielle-book}}] Consider the following problem
$$K - \min\; F(x)=(x_1,x_2), \,\text{ s.t. }\; x\in C$$ where 
$$C:=\left\{x=(x_1,x_2)\in [0,\pi]\times [0,\pi] \,\left| \begin{array}{ll} x_1^2+x_2^2-1-\displaystyle\frac{1}{10}\cos\left(16 \arctan\left(\displaystyle\frac{x_1}{x_2}\right)\right)\geq 0,\\[2.5ex]  (x_1-0.5)^2+(x_2-0.5)^2\leq 0.5.\end{array}\right.\right\}.$$
Consider the map $K:\RR^2\rightrightarrows \RR^2$, given by $$K(x):=\left\{z\in \RR^2: \|z\|_2\le \frac{2}{\min_{i=1,2}x_i}x^Tz\right\}.$$
The {\bf IPG Method} stating at ten random initial points performs as follows: 
{\small
$$\begin{array}{|c|c|c|c|c|c|c|c|c|c|c|}\hline{\bf {Instances}}&{\bf 1}&{\bf 2}&{\bf 3}&{\bf 4}&{\bf 5}&{\bf 6}&{\bf 7}&{\bf 8}&{\bf 9}&{\bf 10}\\\hline\hline  {\bf Initial\, Point}\; 
 & 0.9735&    0.7932&    0.8403&    0.9847  &  0.7508&    0.9786   & 0.9790&   0.9679  &  0.8082   & 0.8965\\
   & 0.6608 &   0.9050 &   0.8664 &   0.6228&    0.9326&    0.6448&    0.6433&    0.6762 &   0.8937 &   0.8046
\\\hline
{\bf Solution}\;\; &
 0.9011 &   0.7407 &   0.7854&    0.9096&    0.7004&    0.9050&    0.9054&   0.8967  &  0.7551 &   0.7754\\&
    0.5589\,&    0.7916\,&    0.7541\,&    0.5228\,&    0.8182\,&    0.5437\,&    0.5423\,&    0.5735 \,&   0.7806 \,&   0.5859\,\\\hline
{\bf CPU\; Time}&  0.0938  &  0.0313 &   0.0625 &   0.0313 &   0.0469 &   0.0156&    0.0313&    0.0313 &   0.0313&    0.0156\\\hline\end{array}$$}\\
{In this case, the maximum number of iterations is achieved.} Nevertheless, good approximations to minimal elements have been computed at each of the ten instances presented in the above table. \end{example}

For all the above examples, we compare the exact and inexact versions by taking the same initial points ($10$ in total on the "x" axes) in term of the CPU time (in seconds on the "y" axes); see figures below.
\begin{figure}[H]
\begin{subfigure}[t]{0.3705\textwidth}
\centering
\includegraphics[width=\textwidth]{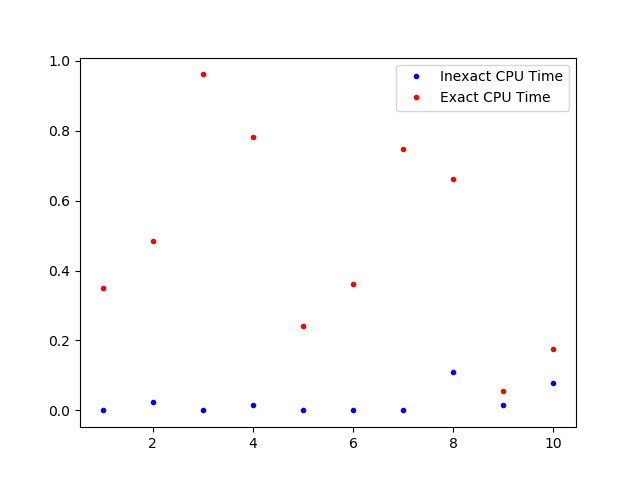}
	 \caption*{\label{fig:pp-twosubspaces1}Example 6.1}
\end{subfigure}\!\!\!\!\!\!\!\!\!\!
\begin{subfigure}[t]{0.3705\textwidth}
\centering
\includegraphics[width=\textwidth]{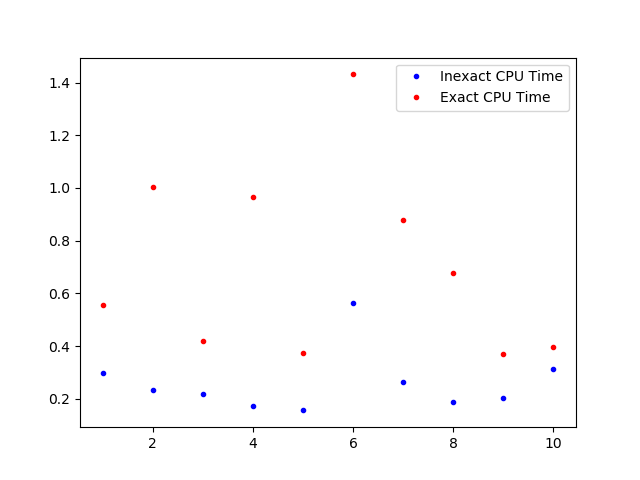}
	 \caption*{\label{fig:pp-twosubspaces2} Example 6.2}	
    \end{subfigure}\!\!\!\!\!\!\!\!\!\!
    \begin{subfigure}[t]{0.3705\textwidth}
    \centering
    \includegraphics[width=\textwidth]{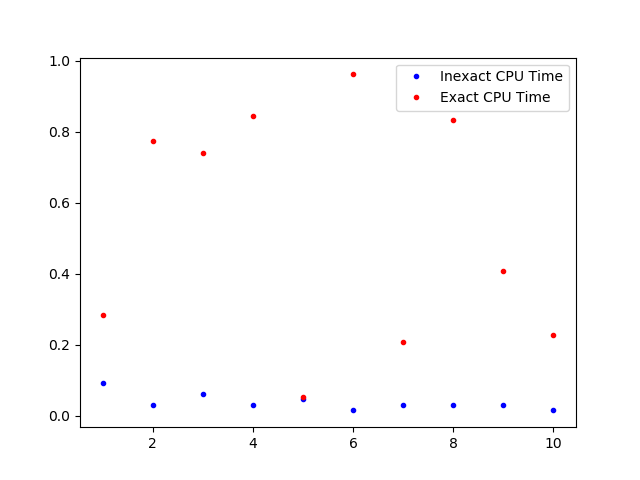}
    	 \caption*{\label{fig:pp-twosubspaces3} Example 6.3}
    \end{subfigure}\\
    
This shows that the inexact versions are significantly faster (CPU Time), in almost all instances, than the exact ones. 
However,  for all initial points of the above examples, the exact versions of the proposed methods take fewer iterations than the inexact ones to converge to a solution. It is worth emphasizing that the exact versions have to solve exactly the harder subproblems $P_{x^k}$ and $Q_{x^k}$ to find the descent direction at each iteration $k$. This is a serious drawback for the computational implementation point of view (avoidable for the above examples), making the exact implementation inefficient in general.        
\end{figure}
\section{Final Remarks}The projected gradient method is one of the classical and basic schemes for solving constrained optimization problems. In this paper, we have extended the exact and unconstrained scheme proposed in \cite{[19]}. The proposed scheme now solves smooth and constrained vector optimization problems under a variable ordering by taking inexact descent directions. This inexact projected approach promotes future research on other efficient variants for these kinds of problems. 
As it is shown in the examples above it is more effective and implementable than the exact one. Moreover, constrained variable optimization problems can now be solved by using the proposed methods.  
However, the full convergence of the generated sequence to a weakly efficient solution is still an open problem in variable order settings. 

 Another important solution concept is non-domination, which is used for optimizing set-valued maps under the variable structure. A variant of the proposed approach can also be considered using the condition $J_F(x^k)v\in -\cup_{x\in C}K(x)$ as the descent direction at each step $k$. However,
this approach will converge to a stationary point as in the case we have proposed. Non-dominated
points are also minimizers under well-known conditions. A future research direction would be to find, in this
case, how we can guarantee that the algorithm converges to non-dominated points.  Set-valued optimization is also an interesting topic and finding algorithms that compute points that at least fulfill conditions, such as those presented in Theorem $3.3$ in \cite{durea}, is important.  

Future work will be also addressed to investigate for some particular instances of this problem the case in which the objective function is a non-smooth function, extending the projected subgradient method proposed in \cite{yunier-2013} to the variable order setting. 
The numerical behavior of these approaches under $K$--convexity of the non-smooth objective function remains open and it is a promising direction to be investigated. Despite its computational shortcomings, it hopefully sets the foundations of future research into more efficient and general algorithms for this setting.

\medskip

\noindent {\large \bf Acknowledgements:} JYBC was partially supported by the National Science Foundation (NSF) Grant DMS - 1816449 and by Research \& Artistry grant from NIU.
The results of this paper were developed during the stay of the second author at the  University of Halle-Wittenberg supported by the Humboldt Foundation. 

%
\end{document}